\documentclass[12pt,reqno]{amsart}
\usepackage{amsfonts}
\usepackage{amsthm}
\usepackage{amsmath}
\usepackage{amscd}
\numberwithin{equation}{section}
\usepackage{hyperref}
\hypersetup{nesting=true,debug=true,naturalnames=true}
\usepackage{graphicx,amssymb,upref}
\usepackage{enumerate}
\usepackage[all]{xy}
\usepackage{color}
\usepackage{mathrsfs}
\usepackage{caption}
\usepackage{paralist}
\usepackage{color}
\usepackage{float}
\usepackage{tikz,pgf}
\usetikzlibrary{calc}
\usetikzlibrary{intersections,decorations.markings}
\usepackage{tikz-cd}
\usepackage{lineno}

\theoremstyle{definition}
\newtheorem{theorem}{\bf Theorem}[section]

\newtheorem{lemma}[theorem]{\bf Lemma}
\newtheorem{corollary}[theorem]{\bf Corollary}

\theoremstyle{definition}

\newtheorem{definition}[theorem]{\bf Definition}
\newtheorem{remark}[theorem]{\bf Remark}

\newtheorem{proposition}[theorem]{\bf Proposition}
\newtheorem{question}[theorem]{\bf Question}
\newcommand{\mm}[1]{\mathrm{#1}}
\newcommand{\mb}[1]{\mathbb{#1}}
\newcommand{\mc}[1]{\mathcal{#1}}
\makeatletter
\@namedef{subjclassname@2020}{\textup{2020} Mathematics Subject Classification}
\makeatother

\begin{document}

\title[On the topology and geometry of certain $13$-manifolds]{On the topology and geometry of certain $13$-manifolds}

\author[Wen Shen]{Wen Shen}
\address{Department of Mathematics, Capital Normal University,
Beijing, P.R.China  }
\email{shenwen121212@163.com}

\begin{abstract}
This paper gives the classifications of certain manifolds $\mc{M}$ of dimension $13$ up to diffeomorphism, homeomorphism, and homotopy equivalence, whose cohomology rings are isomorphic to $H^\ast(\mm{CP}^3\times S^7;\mb{Z})$. Moreover, we prove that either $\mc{M}$ or $\mc{M}\#\Sigma^{13}$ admits a metric of non-negative sectional curvature where $\Sigma^{13}$ is a certain exotic sphere of dimension 13.  	
\end{abstract}


\maketitle
%\tableofcontents

\section{Introduction}\label{intro}

  Classification of the manifolds with a given cohomology ring (up to diffeomorphism, homeomorphism, or homotopy) is one of the central problems in geometric topology. The Poincar\'{e} Theorem is a special case, i.e. for $n\ge 3$, if a $n$-dimensional manifold has the same homology groups as the $n$-dimensional sphere $S^n$, then it is homeomorphic to $S^n$. For a more complicated cohomology ring, there also exist many investigations. In the general case, Wall \cite{Wall1962,Wall1967} investigated $(s-1)$-connected $2s$-manifolds and $(s-1)$-connected $(2s+1)$-manifolds. For concrete cases with specified dimensions, we also list some results. Barden \cite{Barden} classified simply connected 5-manifolds. Kreck and Su \cite{KreckSu} considered certain nonsimply connected 5-manifolds. One of their results is the classification for closed oriented $5$-manifolds $M$ with $\pi_1(M)$ a free group and $H_2(M;\mb{Z}) = 0$. In \cite{Wall1966} Wall classified closed, oriented, simply connected, spin $6$-manifolds with torsion-free homology. Jupp extended the Wall's classification in nonspin case \cite{Jupp}. In dimension $7$, there is a family of smooth manifolds with positive sectional curvature, namely Aloff-Wallach spaces \cite{Aloff}. Kreck and Stolz \cite{KS1, KS} gave a classification for the $7$-manifolds modeled on Aloff-Wallach spaces. Subsequently, a homotopy classification of Aloff-Wallach spaces was given by Kruggel \cite{Kruggel}. Besides, there are other classifications for $7$-manifolds with certain cohomology rings, such as \cite{CroNor,K2018}. 
  %More recently, Fang and Shen \cite{FangShen} considered the $13$-manifolds, named Bazaikin spaces \cite{Ba}.   
  
In this paper, we consider simply connected, closed, smooth $13$-dimensional manifolds, namely $\mathcal M$, with cohomology ring $H^\ast(\mathcal M;\mb{Z})\cong H^\ast(\mm{CP}^3\times S^7;\mb{Z})$, where $\mm{CP}^3$ is the complex projective space. Many examples of these $13$-manifolds arise from the following construction \cite{WZ}: 
From an $S^1$ action on $S^7\times S^7$ given by 
$\theta \circ (x,y) = (e^{il\theta}x,e^{-ik\theta}y)$ with $\mm{gcd}(k,l)=1$,
we get the $S^1$-bundle $S^1\to S^7\times S^7 {\to }  \mathcal{M}$.
From \cite{DDGR}, we can see some interesting geometric properties that there are infinitely many circle quotients $\mathcal M$ of $S^7 \times S^7$ admitting a metric of $\mm{Ric}_2 > 0$. 
In addition, Kerin \cite{Kerin} gave an example $\mathcal M$ with almost positive sectional curvature, which is an $S^7$-bundle over $\mm{CP}^3$.

To investigate the geometric properties of these manifolds $\mathcal M$ further, we need to classify them up to diffeomorphism first. However, for the classification of all manifolds $\mathcal M$,
there are still many unsolved technical issues that will be pointed out in Section \ref{computation}. 

Fortunately, making some constraints to the topology of manifolds $\mathcal M$, we have some interesting conclusions. 

 Throughout the rest of the paper, let $H^i(-)$ (or $H_i(-)$) denote the integral cohomology (or homology) group $H^i(-;\mathbb{Z})$ (or $H_i(-;\mathbb{Z})$) unless other statements are given. We always choose the cup product $x^2$ as the generator of $H^4(\mathcal M)$, where $x$ is a generator of $H^2(\mathcal M)$. Hence we only use an integer to denote the first Pontrjagin class $p_1(\mathcal M)$ of $\mathcal M$. 
 
  By [\cite{Hatcher}, Thoerem 4.57], the cohomology classes in $H^2(\mathcal M)$ are in one-to-one correspondence to the homotopy classes $[\mathcal M, \mm{K}(\mb{Z},2)]$  of maps $\mathcal M\to \mm{K}(\mb{Z},2)$ where $\mm{K}(\mb{Z},2)$ is the Eilenberg-MacLane space. Indeed $\mm{K}(\mb{Z},2)$ is homotopy equivalent to $\mm{CP}^\infty$ that is the colimit over the sequence $\mm{CP}^1\hookrightarrow \cdots \hookrightarrow \mm{CP}^n\hookrightarrow \mm{CP}^{n+1}\hookrightarrow \cdots$ where $\mm{CP}^n\hookrightarrow \mm{CP}^{n+1}$ is the natural inclusion. So we can regard the $x$ as a map $x:\mathcal M\to \mm{CP}^\infty$.  

\begin{definition}
	  $\mathcal M$ admits a {\it restriction lift} if there exists a map $f:\mc{M}\to \mm{CP}^4$ such that the following diagram is homotopy commutative
	\[
\xymatrix@C=0.8cm{
&&\mm{CP}^4\ar[d]^-{\mm{eb}}\\
\mathcal M \ar[rr]^-{x}\ar@{.>}[rru]^-{f}& &  \mm{CP}^\infty
}
\] 
i.e. $\mm{eb}\circ f\simeq x$ where $\mm{eb}$ is the natural inclusion, $x$ denotes a generator of $H^2(\mc{M})$. 
\end{definition}

Obviously, if $\mathcal M_1$ is homotopy equivalent to a manifold $\mathcal M_2$ with a {\it restriction lift}, then $\mathcal M_1$ also admits a {\it restriction lift}.

Let $\overline{\mathcal M}$ be the total space of an $S^7$-bundle over $\mm{CP}^3$. It is easy to check that $\overline{\mathcal M}$ is one of considered manifolds, the composition of the bundle projection $\overline{\mathcal M}\to \mm{CP}^3$ and the natural inclusion $\mm{CP}^3\hookrightarrow \mm{CP}^4$ is a {\it restriction lift} of $\overline{\mathcal M}$. Hence the total space of arbitrary $S^7$-bundle over $\mm{CP}^3$ admits a {\it restriction lift}. 

Until now, I have not found any manifold $\mathcal M$ that really has no {\it restriction lifts}. I guess that all manifolds $\mathcal M$ admit {\it restriction lifts}. If one proves this conjecture or gives a new invariant to characterize whether a manifold $\mathcal M$ admits a {\it restriction lift}, then the classification for nonspin manifolds $\mathcal M$ can be completed by the main theorem of this paper.

Before presenting the main theorem, we introduce $13$-dimensional homotopy spheres, namely $\Sigma^{13}$. A homotopy sphere $\Sigma^{13}$ is a smooth manifold and homotopy equivalent to the standard sphere $S^{13}$. By Poincar\'{e} Theorem, $\Sigma^{13}$ is homeomorphic to $S^{13}$. We say that $\Sigma^{13}$ is an exotic sphere if $\Sigma^{13}$ is not diffeomorphic to $S^{13}$. Up to diffeomorphism, there are three $13$-dimensional homotopy spheres \cite{KervaMilnor}. Moreover, the set of the diffeomorphism classes of $13$-dimensional homotopy spheres and the connected sum $\#$ as addition form an abelian group that is isomorphic to $\mb{Z}_3$. The standard sphere $S^{13}$ represents the $0$ element, arbitrary exotic sphere $\Sigma^{13}$ represents a generator of the group. 

Now, we are ready to give the main theorem of this paper.
\begin{theorem}\label{1.1}
{\it	Let $\mathcal M$, $\mathcal M^\prime$ be nonspin and admit {\it restriction lifts}.  
		\item (1) There exists a homotopy sphere $\Sigma^{13}$ so that $\mathcal M$ is diffeomorphic to $\mathcal M^\prime \# \Sigma^{13}$ if and only if $p_1(\mathcal M)=p_1(\mathcal M^\prime)$;
		\item (2) If $p_1(\mathcal M)=p_1(\mathcal M^\prime)\ne 0\mod 3$, $\mathcal M$ is diffeomorphic to $\mathcal M^\prime$.
		\item (3) $\mathcal M$ is homeomorphic to $\mathcal M^\prime$ if and only if $p_1(\mathcal M)=p_1(\mathcal M^\prime)$.
		\item (4) $\mathcal M$ is homotopy equivalent to $\mathcal M^\prime$ if and only if $p_1(\mathcal M)=p_1(\mathcal M^\prime)\mod 24$.}
\end{theorem}

By computing the first Pontrjagin classes of the total spaces of certain $S^7$-bundles over $\mm{CP}^3$ (see Section \ref{proofmain}), we know that for every integer $n$, there exists an $S^7$-bundle over $\mm{CP}^3$ so that the first Pontrjagin class of its total space $\overline{\mathcal M}$ equals to $nx^2\in H^4(\overline{\mathcal M})$. Thus, by applying (1), (2) of Theorem \ref{1.1}, we have 
\begin{corollary}\label{1.2}
{\it Let $\mathcal M$ be nonspin and admit a {\it restriction lift}, then there exists a homotopy sphere $\Sigma^{13}$ so that $\mathcal M\# \Sigma^{13}$ is diffeomorphic to the total space of an $S^7$-bundle over $\mm{CP}^3$. If $\mc{M}$ additionally satisfies $p_1(\mc{M})\ne 0 \mod 3$, then it is diffeomorphic to the total space of an $S^7$-bundle over $\mm{CP}^3$.}	
\end{corollary}

%The detailed proof for Corollary \ref{1.2} will be shown in Section \ref{proofmain}. 

In Riemannian geometry, many well-known results show that the curvature of a manifold can control its partial topological properties, such as the most fundamental theorems, Bonnet-Myers Theorem and Synge Theorem.
 On the other hand, 
one of the most interesting problems in Riemannian geometry is whether a smooth manifold admits a metric of non-negative sectional curvature or positive sectional curvature. Here, we use the topology of manifolds to investigate the metrics.
%By Corollary \ref{1.2} and Lemma \ref{metric}, we have
\begin{theorem}\label{1.3}
	{\it Let $\mathcal M$ be nonspin and admit a restriction lift with $p_1(\mc{M})\ge 4$, then there exists a homotopy sphere $\Sigma^{13}$ so that $\mc{M}\#\Sigma^{13}$ admits a metric of non-negative sectional curvature. If $\mc{M}$ additionally satisfies $p_1(\mc{M})\ne 0 \mod 3$, it admits a metric of non-negative sectional curvature.}
\end{theorem}

For the general spin case, the same theorem as Theorem \ref{1.1} is unclear. For the spin manifolds $\mathcal M=\mm{CP}^3\times \Sigma^7$ where $\Sigma^7$ is a homotopy $7$-sphere, we know that $\mm{CP}^3\times \Sigma^7$ is diffeomorpic to $\mm{CP}^3\times S^{7}$ where $S^7$ is the standard $7$-sphere, which has been proved in [\cite{Bro}, the proof of Theorem 6.1, pp.37]. Thus,
	for every homotopy sphere $\Sigma^7$, $\mm{CP}^3\times \Sigma^7$ admits a metric of non-negative sectional curvature. 
	
	We mainly use surgery theory \cite{K} to prove (1), (2) and (3) of Theorem \ref{1.1}. In Section \ref{normal5}, we establish a suitable bordism theory for considered manifolds in the smooth category. In Section \ref{obstruction}, we consider the surgery obstruction. Then, we compute a certain bordism group in Section \ref{computation}. To classify  considered manifolds up to homotopy equivalence, we give a model in Section \ref{homosec}. The proofs of Theorem \ref{1.1} and Corollary \ref{1.2} are presented in Section \ref{proofmain}. At last, in Section \ref{metricSection}, we discuss the metrics of the sphere bundles of certain $8$-dimensional vector bundles and prove Theorem \ref{1.3}.  

\section{The normal $5$-smoothing}\label{normal5}
Let $\pi: {B}\to \mm{BO}$ be a fibration over the classifying space $\mm{BO}$. Given a smooth manifold $M$, we call the classifying map of the stable normal bundle of $M$ as stable normal Gauss map, denoted by $\mathscr{N}_{M} :M\to \mm{BO}$. 
As described in [\cite{K}, Definition, pp.711], a normal ${B}$-structure of $M$ is a lift $ \bar{\mathscr{N}}_{M}: M\to {B}$ of $ {\mathscr{N}}_{M}$. $\bar {\mathscr{N}}_{M}$ is called a normal $(B,\ell)$-smoothing when it is also a $(\ell + 1)$-equivalence, i.e. the induced homomorphism $\bar{\mathscr{N}}_{M\ast}:\pi_n(M)\to \pi_n(B)$ is an isomorphism for $n<\ell+1$, an epimorphism for $n=\ell+1$.

Let $\mathcal M$ be a simply connected, closed, smooth $13$-dimensional manifold with cohomology ring $H^\ast(\mathcal M)\cong H^\ast(\mm{CP}^3\times S^7)$.
Assume that $\mc{M}$ admits a {\it restriction lift} $f:\mc{M}\to \mm{CP}^4$. Note that the first Pontryagin class $p_1(\mathcal{M})$ of $\mathcal M$ equals to $sx^2\in H^4(\mathcal M)$, 
the second Stiefel-Whitney class $w_2(\mathcal{M})$ is $s x\mod{2}\in H^2(\mathcal M;\mb{Z}_2)$ for some integer $s$. Next we construct a vector bundle $\xi$ over $\mm{CP}^4$ so that its first Pontryagin class depends on the integer $s$.

Let $\mathcal{H}$ denote the Hopf bundle over $\mathrm{CP}^4$. If $s\ge 0$, let $\xi$ be the complementary bundle of the Whitney sum $s\mathcal{H}$; if $s<0$, let $\xi$ be the Whitney sum $-s\mathcal{H}$. Formally $\xi=-s\mc{H}$. Note that the first and second Chern classes of $\xi$ are 
$c_1(\xi)=-sx$, $c_2(\xi)=\frac{s(s+1)}{2}x^2$,
the first Pontryagin class $p_1(\xi)=-sx^2$. 

Let $\mm{BO}\langle n \rangle$ be the $(n-1)$-connected cover of $\mm{BO}$, i.e. $\mm{BO}\langle n \rangle$ is $(n-1)$-connected, there exists a fibration $\mm{BO}\langle n \rangle \to \mm{BO}$ inducing isomorphisms on $\pi_k$ for $k\ge n$. It is well-known that $\mm{BSO}=\mm{BO}\langle 2\rangle$, $\mm{BSpin}=\mm{BO}\langle 4\rangle$ and $\mm{BString}=\mm{BO}\langle 8\rangle$. Furthermore, there is sequence $\mm{BO}\langle 8\rangle\to \mm{BO}\langle 4\rangle\to \mm{BO}\langle 2\rangle\to \mm{BO}$ where each map is a fibration. From the unstable variation of the Postnikov tower in [\cite{BuNa}, pp.44], we get the obstructions of  a map $X\to \mm{BO}$ lift to $\mm{BO}\langle n\rangle$ for $n=2,4,8$. Indeed a map $X\to \mm{BO}$ is the classifying map of a stable vector bundle $\eta$ over $X$. If $w_2(\eta)=w_1(\eta)=0$, $H^4(X)$ is torsion-free, $p_1(\eta)=0$, then the map has a lift to $\mm{BO}\langle 8\rangle$.

Let $\mathscr{N}_{\mc{M}}$ denote both the stable normal Gauss map $\mc{M}\to \mm{BO}$ and the stable normal bundle of $\mc{M}$. By computing the
first Pontryagin class and the second Stiefel-Whitney class of the difference bundle $\mathscr{N}_{\mc{M}}-f^\ast \xi$, we have $w_2(\mathscr{N}_{\mc{M}}-f^\ast \xi)=0$ and $p_1(\mathscr{N}_{\mc{M}}-f^\ast \xi)=0$. Recalling the cohomology ring of $\mathcal M$, we have that the classifying map of the difference bundle has a lift to $\mm{BO}\langle 8\rangle$, namely $\nu:\mathcal M\to \mm{BO}\langle 8\rangle$. Let $\gamma_8$ be the universal bundle over $\mm{BO}\langle 8 \rangle$, i.e. the pull-back bundle of the universal bundle over $\mm{BO}$ through the fibration $\mm{BO}\langle 8 \rangle\to \mm{BO}$. Then we have $\nu^\ast\gamma_8\cong \mathscr{N}_{\mc{M}}-f^\ast \xi$.

Let ${B}=\mm{BO}\langle 8 \rangle \times  \mm{CP}^4$. There is a product bundle, namely $\gamma_8\times \xi$, over $B$.
Its classifying map, denoted by
$\pi(\xi): {B}\to \mm{BO}$, can be factored as the composition $B\hookrightarrow E_{\pi} \to \mm{BO}$ of a homotopy equivalence and a fibration by [\cite{Hatcher}, Proposition 4.64, pp.407]. Hence we regard the map $\pi(\xi)$ as a fibration.

Now we claim that the composition 
$$g:\mc{M}\stackrel{\Delta}{\longrightarrow }\mc{M}\times \mc{M}\stackrel{\nu\times f}{\longrightarrow}\mm{BO}\langle 8 \rangle \times \mm{CP}^4$$ is a normal ${B}$-structure of $\mc{M}$ where $\Delta$ is the diagonal map. Recalling the Whitney sum of vector bundles, we have $\mathscr N_{\mc{M}}\cong \nu^\ast\gamma_8\oplus f^\ast\xi\cong g^\ast(\gamma_8\times \xi)$. By the definition of normal $B$-structure, the claim follows. Indeed it is easy to check that $g$ induces isomorphisms on $H_n$ for $0\le n\le 6$. So $g$ is a normal $(B,5)$-smoothing of $\mc{M}$.

All $n$-dimensional manifolds with normal ${B}$-structures under the cobordant relation and the disjoint union as operation form a bordism group $\Omega_n^{\mm{O}\langle 8 \rangle}(\xi)$ [\cite{Sw}, pp.226] which is isomorphic to the homotopy group $\pi_n(\mm{M{B}})$ by the Pontryagin-Thom isomorphism. Here, $\mm{M{B}}=\mm{MO}\langle 8 \rangle\wedge \mm{M}\xi $ where $\mm{M}\xi$ and $\mm{MO}\langle 8\rangle$ are the Thom spectra of the bundles $\xi$ and $\gamma_8$.
A bordism class of $\Omega_n^{\mm{O}\langle 8 \rangle}(\xi)$ is denoted by a pair $[M,\bar{\mathscr N}_M]$ where $M$ is a $n$-dimensional manifold, $\bar{\mathscr N}_M$ is a normal $B$-structure of $M$. 
Here we need to note that the bordism group $\Omega_n^{\mm{O}\langle 8 \rangle}(\xi)$ depends on the bundle $\xi$. Let $\mc{M}$, $\mc{M}'$ admit {\it restriction lifts}. When $p_1(\mc{M})=p_1(\mc{M}')$, we can take the same bundle $\gamma_8\times \xi$ over $B$ such that the bordism classes $[\mc{M},g]$ and $[\mc{M}',g']$ lie in the same bordism group $\Omega_{13}^{\mm{O}\langle 8 \rangle}(\xi)$.
In Section \ref{computation}, we will consider whether the bordism classes are equal.

\section{Surgery obstruction}\label{obstruction}

Recall $B=\mm{BO}\langle 8 \rangle \times \mm{CP}^4$ and the bundle $\gamma_8\times \xi$ over it. 
In this section, we assume that $\mc{M}$ and $\mc{M}'$ admit {\it restriction lifts}, $p_1(\mc{M})=p_1(\mc{M}')$. Besides,
assume that the bordism class $[\mc{M},g]$ equals to $[\mc{M}',g']$ in $\Omega_{13}^{\mm{O}\langle 8 \rangle}(\xi)$ where $g,g'$ are the normal $(B,5)$-smoothings of $\mc{M},\mc{M}'$, defined in Section \ref{normal5}.

 Then there exists a $14$-dimensional compact manifold $W$ with a normal $B$-structure $\phi:W\to B$ such that its boundary is the disjoint union $\mathcal M\sqcup -\mathcal M^\prime$ where the orientation of $-\mathcal M^\prime$ is opposite to $\mathcal M^\prime$, the composition of the embedding $\mc{M}\to W$ (resp. $\mc{M}'\to W$) and $\phi$ is homotopy equivalent to the normal $(B,5)$-smoothing $g$ (resp. $g'$). 

 From surgery theory [\cite{K}, Section 3, Corollary 1], we can assume that $\phi:W\to  B$ is a $7$-equivalence. So the embedding 
$i:\mathcal M\to W$ induces isomorphisms $H_n(\mathcal M)\cong H_n(W)$ for $0\le n\le 6$. Applying the relative Poincar\'e Duality, i.e. $H^n(W,\mathcal M)\cong H_{14-n}(W,\mathcal M')$, the (co)homology long exact sequences for the pairs $(W,\mathcal M)$ and $(W,\mathcal M')$, we have that $H_7(W)\cong H_7(W,\mathcal M)\oplus H_7(\mathcal M)$ is torsion-free. 

%By the relative Poincar\'e Duality,
%$$ \langle H^7(W,\mathcal{M})\cup H^7(W,\mathcal{M}'), [W,\partial W]  \rangle$$
%we have an antisymmetric quadratic form as follows
%$$\lambda:H^7(W,\partial W)\otimes H^7(W,\partial W)\to \mathbb{Z}$$
%$$\lambda(w_1,w_2)= \langle i_1(w_1)\cup i_2(w_2), [W,\partial W]  \rangle$$
%where
% $i_1:H^7(W,\partial W) \to H^7(W,\mathcal{M})$, $i_2:H^7(W,\partial W) \to H^7(W,\mathcal{M}')$ (here the antisymmetry).
%Note that the triple $(W,\partial W,\mc{M})$ admits the following short exact sequence $0\to H^6(\partial W,\mc{M})\to H^7(W,\partial W)\to H^7(W,\mc{M})\to 0$. By the  antisymmetric quadratic form $\lambda$ and the relative Poincar\'e Duality, we have that $\mm{rank}(H_7(W,\mathcal M))$ is even. 

Considering the surgery obstruction for replacing $W$ with an h-cobordism, we first transform the elements in $H_7(W,\mathcal M)$ into the kernel of the homomorphism $\phi_\ast:\pi_7(W)\to \pi_7(B)$. There is a  commutative diagram as follows.
 \[
\xymatrix@C=.3cm{
&\pi_8(B)=\mb{Z}\ar[r]^-{} \ar[d]^-{}& \pi_8(B,W) \ar[r]^-{} \ar[d]^-{\cong}& \pi_7(W)  \ar[r]^-{\phi_\ast} \ar[d]^-{h_\ast}&\pi_7(B)=0 \ar[d]^-{}\\
H_8(W)=0 \ar[r]& H_8(B)=\mb{Z}^2\ar[r]^-{} &H_8(B,W) \ar[r]^-{}& H_7(W)\ar[r]^-{} &H_7(B)=0
}
\]
 The first (resp. the second) row is the long exact sequence of homotopy (resp. homology) groups for the map $\phi:W\to B$, a down arrow is a Hurewicz homomorphism. By chasing the diagram, we have
\begin{lemma}\label{surgery1}
	$\mm{ker}\phi_\ast=\pi_7(W)\cong H_7(W)\oplus \mb{Z}\cong H_7(W,\mathcal M)\oplus \mb{Z}^2.$
\end{lemma}

By the smooth embedding theorem \cite{Milnor1965} and Lemma \ref{surgery1}, every generator $u\in H_7(W,\mathcal M)$ can be represented by an embedding $S^7\to \mathring W$ where $\mathring W$ is the interior of $W$. From \cite{KervaMilnor},
  any $7$-dimensional vector bundle  over $S^7$ is trivial. Hence, the normal bundle of this embedding is trivial, and thus the total space of the normal bundle is $S^7\times \mb{R}^7$. By the tubular neighborhood theorem, there exists an embedding from the total space of the disk bundle of the normal bundle, i.e. $S^7\times D^7$ to $\mathring W$. We denote the embedding by 
   $\bar{u}:S^7\times D^7\to \mathring{W}$. Furthermore, $u=\bar{u}|_{S^7\times \{0\}}$ where $\{0\}$ is the center of $D^7$. Then we define
$$\mathscr W := D^8 \times D^7 \cup_{\bar u} W \times I $$
where we consider $\bar u$ as a map to $W \times \{1\}$. $\mathscr W$ is called the trace of a surgery via $\bar u$. The boundary of $\mathscr W$ is $W \cup(\partial W \times I)\cup W'$ and we call $W'$ the result of a surgery of index $8$ via $\bar u$. More explicitly,
$$W' = D^8\times S^6\cup_{\bar u} (W -\bar{u}(S^7 \times \mathring{D}^7)). $$
Here we note [\cite{K}, pp.715]
\begin{equation}
	\mathscr W\simeq D^8\cup_{u} W\simeq D^7\cup W'. \label{homotopyequiW}
\end{equation}

Now we claim that the normal $B$-structure $\phi:W\to B$ of $W$ extends to a normal $B$-structure of $\mathscr W$. To prove the claim, it is sufficient to
check the condition in [\cite{K}, Section 3, iii) of Lemma 2]. By [\cite{AtiHi}, Theorem 1], there exists a real vector bundle $\theta$ over $S^8$ with $w_8(\theta)\ne 0$. Let $\Theta :S^8\to \mm{BO}$ be the classifying map of the stable vector bundle represented by $\theta$. Then we have $\Theta^\ast(w_8)\ne 0$ where $w_8\in H^8(\mm{BO};\mb{Z}_2)$. Since that the map $\pi(\xi):B\to \mm{BO}$ induces isomorphism on $\pi_8$, the desired result follows. 

Note that the boundary of $ W'$ is also $\mathcal M\sqcup -\mathcal M^\prime$. Therefore, the normal $B$-structure of $\mathscr W$ induces a normal $B$-structure $\phi':W'\to B$ of $ W'$ so that $g\simeq \phi'|_{\mc{M}}$ and $g'\simeq \phi'|_{\mc{M}'}$.   

Finally we show that $\mm{rank}(H_7(W',\mc{M}))=\mm{rank}(H_7(W,\mc{M}))-2$, $H_n(W',\mc{M})=0$ for $n>0$ and $n\ne 7$.
 
 By Equation (\ref{homotopyequiW}) and the relative Poincar\'e Duality, for $1\le n\le 5$ or $9\le n$, we have $H_n(W',\mc{M})=0$.
 
 By $\mathscr W\simeq D^8\cup_u W$, we have that $H_8(\mathscr W)=0$, $H_7(\mathscr W)$ is torsion-free, the embedding $\mc{M}\to \mathscr W$ induces an isomorphism from $H_7(\mc{M})$ to a summand of $H_7(\mathscr W)$. Thus, by $\mathscr W\simeq D^7\cup W'$, we have $H_8( W')=0$, $H_7( W')$ is isomorphic to a summand of $H_7(\mathscr W)$. Furthermore, by the composition $\mc{M}\to W'\to \mathscr W$, $H_7(\mc{M})$ is isomorphic to a summand of $H_7(W')$. Hence $H_8(W',\mc{M})=0$, $H_7(W',\mc{M})$ is torsion-free. 
 
 Similarly we have the same results for the pair $(W',\mc{M}')$ as above. %Note $H_7(W',\mc{M}')$ is torsion-free. 
 By the universal coefficient theorem, we have $H^8(W',\mc{M}')=0$. So $H_6(W',\mc{M})=0$ by the relative Poincar\'e duality. 

 Note that $H_6(\mc{M})\cong H_6(B)$. By
  the composition $\mc{M}\to W'\to B$ and $H_6(W',\mc{M})=0$, we have $H_6(\mc{M})\cong H_6(W')$. Thus, by the homology long exact sequence and the above conclusions for the pair $(W',\mc{M})$, we have $H_7(W')\cong H_7(W',\mc{M})\oplus H_7(\mc{M})$. Now we consider the homology long exact sequence for the pair $(\mathscr W,W')$
  $$0\to H_7(W')\to H_7(\mathscr W)\to H_7(\mathscr W,W')=\mb{Z}\to H_6(W')\to H_6(\mathscr W)$$
By $\mathscr W\simeq D^8\cup_u W$, $\mm{rank}(H_7(\mathscr W))=\mm{rank}(H_7(W))-1$. By the composition $\mc{M}\to W'\to \mathscr W$ and the isomorphisms $H_6(\mc{M})\cong H_6(\mathscr W)$, $H_6(\mc{M})\cong H_6( W')$, we have $H_6(W')\cong H_6(\mathscr W)$. Hence $\mm{rank}(H_7(W'))=\mm{rank}(H_7(\mathscr W))-1=\mm{rank}(H_7(W))-2$. Moreover, $\mm{rank}(H_7(W',\mc{M}))=\mm{rank}(H_7(W,\mc{M}))-2$. 	 

Indeed $R=\mm{rank}(H_7(W,\mc{M}))$ is even. Repeating the surgery as above $R/2$ times, we can get the needed h-cobordism.
   
  \begin{lemma}\label{hcobordism}
  If $[\mathcal M,g]=[\mc{M}',g']\in \Omega_{13}^{\mm{O}\langle 8\rangle}(\eta)$, then there is a smooth $h$-cobordism
  $(W,\mathcal M,\mathcal M^\prime).$
  \end{lemma} 

\section{The computations of certain bordism groups}\label{computation}
 Recall $\Omega_\ast^{\mm{O}\langle 8 \rangle}(\xi)\cong\pi_\ast(\mm{M{B}})$ and $\mm{M{B}}=\mm{MO}\langle 8 \rangle\wedge \mm{M}\xi $ where $\mm{M}\xi$ and $\mm{MO}\langle 8\rangle$ are the Thom spectra of the bundles $\xi$ and $\gamma_8$ (cf. Section \ref{normal5}).

There is an Atiyah-Hirzebruch spectral sequence (AHSS) as follows
 $$E_2^{p,q}\cong H_p(\mm{M}\xi
 ;\pi_q(\mm{MO}\langle 8\rangle))\Longrightarrow \mm{MO}\langle 8\rangle_\ast(\mm{M}\xi)\cong \pi_\ast(\mm{MO}\langle 8\rangle\wedge \mm{M}\xi).$$
 Recall the $E_2$-terms of the Adams spectral sequence (ASS) for $\mm{MO} \langle  8 \rangle $ as in Fig.\ref{Fig.1}. In dimension $i \le 14$ \cite{Giambalvo,HoRa1995}, $\pi_i(\mm{MO} \langle  8 \rangle )$ is as follows.

\begin{center}
\begin{tabular}{|c|c|c|c|c|c|c|c|c|c|c|c|c|c|c|c|}
\hline
$i$ & 0&1& 2& 3&4& 5& 6&7&8&9\\
\hline
$\pi_i(\mm{MO} \langle  8\rangle )$ & $\mathbb{Z}$ & $\mathbb{Z}_2$ & $\mathbb{Z}_2$ & $\mathbb{Z}_{24}$& 0& 0& $\mathbb{Z}_2$ & 0& $ \mathbb{Z}\oplus \mathbb{Z}_2$ & $\mathbb{Z}_2\oplus \mathbb{Z}_2 $\\
\hline

\end{tabular}
\end{center}
\begin{center}
\begin{tabular}{|c|c|c|c|c|c|c|c|c|c|c|c|c|c|c|c|}
\hline
$i$ &10&11&12&13&14\\
\hline
$\pi_i(\mm{MO} \langle  8 \rangle )$  & $\mathbb{Z}_6 $ & 0& $\mathbb{Z}$ & $\mathbb{Z}_3$ & $\mathbb{Z}_2$\\
\hline

\end{tabular}
\end{center}

\begin{figure}
\centering
\begin{tikzpicture}[>=stealth,scale=1,line width=0.5pt]{enumerate}

\pgfmathsetmacro{\ticker}{0.125}

\coordinate [label=225:$0$](A) at (0,0);
\coordinate (B) at (0,5.6);
\coordinate (D) at (10.2,0);
\draw(D)--(A)--(B);

\coordinate [label=left:$s$](E) at ($(B)+(-0.4,-0.2)$);
\coordinate [label=below:$t-s$](F) at ($(D)+(-0.2,-0.4)$);

\foreach \i/\texti  in {1,2,3,4,5,6,7,8,9,10,11,12,13,14,15,16,17} {
\draw (0.6*\i,0) --(0.6*\i,\ticker) node[label=below:\texti]{};
}
\foreach \j/\textj  in {1,2,3,4,5,6,7,8} {
\draw (0,0.7*\j) --(\ticker,0.7*\j) node[label=left:\textj]{};
}
{\tiny
\coordinate[label=right:$h_0$] (I) at (0.1,0.7);
}
{\tiny
\coordinate[label=right:$h^2_0$] (I) at (0.1,1.4);
}
\path[name path = row1](0,0.7)--(10.2,0.7);
\path[name path = row2](0,1.4)--(10.2,1.4);
\draw[name path = row3][dashed] (0,2.1) -- (10.2,2.1);
\draw[name path = xieh1](0,0)--(0.6,0.7)--(1.2,1.4)--(1.8,2.1);
\path[draw,fill,name intersections={of = row1 and xieh1,by=h1}](h1)circle(1pt);
{\tiny
\coordinate[label=right:$h_1$] (I) at (0.7,0.7);
}
\path[draw,fill,name intersections={of = row2 and xieh1,by=h12}](h12)circle(1pt);
{\tiny
\coordinate[label=left:$h^2_1$] (I) at (1.1,1.4);
}
\path[draw,fill,name intersections={of = row3 and xieh1,by=h13}](h13)circle(1pt);
{\tiny
\coordinate[label=above:${h^3_1=h^2_0h_2}$] (I) at (1.8,2.1);
}
\draw[name path = c3] (1.8,0.7) -- (h13);
\path[draw,fill,name intersections={of = c3 and row2,by=h0h2}](h0h2)circle(1pt);
\path[draw,fill,name intersections={of = c3 and row1,by=h2}](h2)circle(1pt);
{\tiny
\coordinate[label=left:${h_2}$] (I) at (1.7,0.8);
}
\draw[name path = xieh2](0,0)--(h2)--(3.6,1.4);
\path[draw,fill,name intersections={of = xieh2 and row2,by=h22}](h22)circle(1pt);
{\tiny
\coordinate[label=left:${h^2_2}$] (I) at (3.5,1.5);
}
\path[name path = row4](0,2.8)--(10.2,2.8);
\path[name path = row5](0,3.5)--(10.2,3.5);
\path[name path = row6](0,4.2)--(10.2,4.2);
\path[name path = row7](0,4.9)--(10.2,4.9);
\path[name path = row8](0,5.6)--(10.2,5.6);
\path[name path = c8](4.8,0)--(4.8,5.6);
\path[draw,fill,name intersections={of = c8 and row3,by=c0}](c0)circle(1pt);
{\tiny
\coordinate[label=left:${c_0}$] (I) at (4.7,2.0);
}
\path[draw,fill,name intersections={of = xieh2 and row2,by=h22}](h22)circle(1pt);
\draw(4.8,2.8)--(4.8,5.6);
\path[draw,fill,name intersections={of = c8 and row4,by=omega}](omega)circle(1pt);
{\tiny
\coordinate[label=left:${\omega}$] (I) at (4.7,2.8);
}
\path[draw,fill,name intersections={of = c8 and row5,by=h0omega}](h0omega)circle(1pt);
\path[draw,fill,name intersections={of = c8 and row6,by=h02omega}](h02omega)circle(1pt);
\path[draw,fill,name intersections={of = c8 and row7,by=h03omega}](h03omega)circle(1pt);
\path[draw,fill,name intersections={of = c8 and row8,by=h04omega}](h04omega)circle(1pt);
\draw[name path = xiec0](4.8,2.1)--(5.4,2.8);
\path[draw,fill,name intersections={of = row4 and xiec0,by=h1c0}](h1c0)circle(1pt);
\draw[name path = xieomega](4.8,2.8)--(6.6,4.9);
\path[draw,fill,name intersections={of = row5 and xieomega,by=h1omega}](h1omega)circle(1pt);
\path[draw,fill,name intersections={of = row6 and xieomega,by=h12omega}](h12omega)circle(1pt);
\path[draw,fill,name intersections={of = row7 and xieomega,by=h13omega}](h13omega)circle(1pt);
{\tiny
\coordinate[label=above:${h_1^3\omega=h^2_0h_2\omega}$] (I) at (6.3,4.9);
}
\draw[name path = c11](h13omega)--(6.6,3.5);
\path[draw,fill,name intersections={of = row5 and c11,by=h2omega}](h2omega)circle(1pt);
\path[draw,fill,name intersections={of = row6 and c11,by=h0h2omega}](h0h2omega)circle(1pt);
{\tiny
\coordinate[label=left:${h_2\omega}$] (I) at (6.6,3.6);
}
\draw[name path = xie2omega](omega)--(8.4,4.2);
\path[draw,fill,name intersections={of = row6 and xie2omega,by=h22omega}](h22omega)circle(1pt);
{\tiny
\coordinate[label=above:${h^2_2\omega=h^2_0d_0}$] (I) at (h22omega);
}
\draw[name path = c12](7.2,2.1)--(7.2,5.6);
\path[draw,fill,name intersections={of = row3 and c12,by=tau}](tau)circle(1pt);
\path[draw,fill,name intersections={of = row4 and c12,by=h0tau}](h0tau)circle(1pt);
\path[draw,fill,name intersections={of = row5 and c12,by=h02tau}](h02tau)circle(1pt);
\path[draw,fill,name intersections={of = row6 and c12,by=h03tau}](h03tau)circle(1pt);
\path[draw,fill,name intersections={of = row7 and c12,by=h04tau}](h04tau)circle(1pt);
\path[draw,fill,name intersections={of = row8 and c12,by=h05tau}](h05tau)circle(1pt);
{\tiny
\coordinate[label=left:${\tau}$] (I) at (7.1,2);
}
\draw[name path = c14](h22omega)--(8.4,2.8);
\path[draw,fill,name intersections={of = row4 and c14,by=d0}](d0)circle(1pt);
\path[draw,fill,name intersections={of = row5 and c14,by=h0d0}](h0d0)circle(1pt);
{\tiny
\coordinate[label=below:${d_0}$] (I) at (8.3,2.8);
}
\draw[name path = xied0](d0)--(9,3.5);
{\tiny
\coordinate[label=right:${h^2_0\kappa=h_1d_0}$] (I) at (9,3.5);
}
\path[draw,fill,name intersections={of = row5 and xied0,by=h1d0}](h1d0)circle(1pt);
\draw[name path = c15](h1d0)--(9,2.1);
\path[draw,fill,name intersections={of = row3 and c15,by=kappa}](kappa)circle(1pt);
\path[draw,fill,name intersections={of = row4 and c15,by=h0kappa}](h0kappa)circle(1pt);
{\tiny
\coordinate[label=left:${\kappa}$] (I) at (8.9,2);
}
\draw[dashed](7.2,2.1)--(6.6,3.5);
\draw[dashed](7.2,2.8)--(6.6,4.2);\draw[dashed](7.2,3.5)--(6.6,4.9);
\draw[dashed](9,2.1)--(8.4,3.5);\draw[dashed](9,2.8)--(8.4,4.2);

\end{tikzpicture}
\renewcommand{\figurename}{Fig}
 \caption{ $\mm{Ext}^{s,t}_{\mathscr{A}_2}(\mathbb{Z}_2, \mathbb{Z}_2)$}
  \label{Fig.1}
\end{figure}

In the AHSS for $\Omega_{13}^{\mm{O}\langle 8\rangle}(\xi)$, all nontrivial $E_2$-terms are $E_2^{0,13}=\mb{Z}_3$ and $E_2^{4,9}=\mb{Z}_2^2$.
Let $G=E_\infty^{4,9}$. Obviously the $2$-primary part ${_2\Omega}_{13}^{\mm{O}\langle 8\rangle}(\xi)={_2\pi}_{13}(\mm{MB})={\pi}_{13}(\mm{MB})/K^2$ is isomorphic to $G=\mb{Z}_2^n$ where $n=0,1$ or $2$ where $K^2$ is a subgroup of all elements of finite order prime to $2$. Next, we will observe the $G$ in the following Adams spectral sequence (ASS) 
$$E_2^{s,t}(\mm{ASS})=\mm{Ext}_{\mathscr A}^{s,t}(H^\ast(\mm{MB};\mb{Z}_2),\mb{Z}_2)\Longrightarrow {_2\pi}_{t-s}(\mm{MB})$$
where $\mathscr A$ is the mod $2$ Steenrod algebra. Note that $H^\ast(\mm{MB};\mb{Z}_2)\cong H^\ast(\mm{MO}\langle 8\rangle ;\mb{Z}_2)\otimes H^\ast(\mm{M}\xi;\mb{Z}_2)$.
By \cite{Giambalvo}, as a module over $\mathscr{A}$ in dimensions $<16$, $H^\ast(\mm{MO}\langle 8 \rangle; \mathbb{Z}_2)$ is isomorphic to the quotient $\mathscr{A}\slash\slash\mathscr{A}_2$  of $\mathscr{A}$ by the left ideal generated by $\mathscr{A}_2$, where $\mathscr{A}_2$ is the Hopf subalgebra of $\mathscr{A}$ generated by $\mathrm{Sq}^1$, $\mathrm{Sq}^2$ and $\mathrm{Sq}^4$.
Therefore, $E_2^{s,t}(\mm{ASS})\cong \mm{Ext}_{\mathscr A_2}^{s,t}(H^\ast(\mm{M}\xi;\mb{Z}_2),\mb{Z}_2)$ in the range $t-s<15$ by the change-of-rings theorem.
%$$\mm{Ext}^{s,t}=\mm{Ext}_{\mathscr A}^{s,t}(H^\ast(\mm{MB};\mb{Z}_2),\mb{Z}_2)\cong \mm{Ext}_{\mathscr A_2}^{s,t}(H^\ast(\mm{M}\xi;\mb{Z}_2),\mb{Z}_2)$$
By Thom isomorphism, $H^\ast(\mm{M}\xi;\mb{Z}_2)\cong H^\ast(\mm{CP}^4;\mb{Z}_2)\cup U$ as $\mb{Z}_2$-vector space where $U$ is the stable Thom class of the bundle $\xi$. In particular, for any $a\in H^\ast(\mm{CP}^4;\mb{Z}_2)$
$$\mm{Sq}^n(aU)=\Sigma_{i=0}^n\mm{Sq}^ia\cup \mm{Sq}^{n-i}U,\quad \mm{Sq}^iU=w_i(\xi) U$$
where $w_i(\xi)$ is the $i$th Stiefel-Whitney of bundle $\xi$. 

Let $\mm{Ext}^{s,t}=\mm{Ext}_{\mathscr A_2}^{s,t}(H^\ast(\mm{M}\xi;\mb{Z}_2),\mb{Z}_2)$.
There is an Algebraic Atiyah-Hirzebruch spectral sequence (AAHSS) (see \cite{WangXu}, Section 5) as follows
$$E_1^{s,m,n}=\mm{Ext}_{\mathscr{A}_2}^{s,s+m}(\mathbb{Z}_2,\mathbb{Z}_2)\otimes H^n(\mm{M}\xi;\mb{Z}_2)\Longrightarrow \mm{Ext}^{s,s+m+n}$$
$\mm{Ext}_{\mathscr{A}_2}^{s,s+m}(\mathbb{Z}_2,\mathbb{Z}_2)$ is computed in \cite{Giambalvo} and shown in Fig.\ref{Fig.1}. One dark dot in Fig.\ref{Fig.1} denotes a generator of the Ext-group, a vertical line $|$ between two dots denotes an $h_0$-action (cf. the product in $\mm{Ext}$-group [\cite{Adams1960}, pp.30]), a slash $/$ with slope $1$ between two dots denotes an $h_1$-action, a slash $/$ with slope ${1}/{3}$ denotes an $h_2$-action, a dashed line represents a differential in the ASS for $\pi_\ast(\mm{MO}\langle 8\rangle)$ (see \cite{Giambalvo}).

\begin{figure}
\centering
\begin{tikzpicture}[>=stealth,scale=1,line width=0.5pt]{enumerate}

\pgfmathsetmacro{\ticker}{0.125}

\coordinate [label=225:$0$](A) at (0,0);
\coordinate (B) at (0,5.6);
\coordinate (D) at (10.2,0);
\draw(D)--(A)--(B);

\coordinate [label=left:$s$](E) at ($(B)+(-0.4,-0.2)$);
\coordinate [label=below:$t-s$](F) at ($(D)+(-0.2,-0.4)$);

\foreach \j/\textj  in {1,2,3,4,5,6,7,8} {
\draw (0,0.7*\j) --(\ticker,0.7*\j) node[label=left:\textj]{};
}
\path[name path = x](D)--(A);
\path[name path = row1](0,0.7)--(10.2,0.7);
\path[name path = row2](0,1.4)--(10.2,1.4);
\draw[name path = row3][dashed] (0,2.1) -- (10.2,2.1);
\path[name path = row4](0,2.8)--(10.2,2.8);
\path[name path = row5](0,3.5)--(10.2,3.5);
\path[name path = row6](0,4.2)--(10.2,4.2);
\draw[name path = row7][dashed](0,4.9)--(10.2,4.9);
\path[name path = row8](0,5.6)--(10.2,5.6);
\path[name path = c1] (0.6,0) -- (0.6,5.6);
\draw[name path = c11] (0.6,0) -- (0.6,0.15);
%\path[fill,name intersections={of = c11 and x,by=12}];%(12)circle(1pt)
{\tiny
\coordinate[label=below:$12$] (I) at (0.6,-0.1);
}
\path[name path = c2] (3.8,0) -- (3.8,5.6);
\path[name path = c22] (3.6,0) -- (3.6,5.6);
\draw[name path = c21] (3.8,0) -- (3.8,0.15);
%\path[fill,name intersections={of = c21 and x,by=13}](13)circle(1pt);
{\tiny
\coordinate[label=below:$13$] (I) at (3.8,-0.1);
}
\path[name path = c3] (7,0) -- (7,5.6);
\draw[name path = c31] (7,0) -- (7,0.15);
\path[fill,name intersections={of = c31 and x,by=14}];
{\tiny
\coordinate[label=below:$14$] (I) at (7,-0.1);
}
\path[name path = c4](1,0)--(1,5.6);
%\path[fill,name intersections={of = c4 and row6,by=384}](384)circle(1pt);
%{\tiny
%\coordinate[label=right:${h_1^2\omega x U}$] (I) at (384);
%}
%\path[fill,name intersections={of = c4 and row4,by=d0U}](d0U)circle(1pt);
%{\tiny
%\coordinate[label=right:$d_0U$] (I) at (d0U);
%}
%\path[name path = c7](6.9,0)--(6.9,5.6);
%\path[fill,name intersections={of = c7 and row6,by=6104}](6104)circle(1pt);
%{\tiny
%\coordinate[label=left:$h_1^2\omega x^2 U$] (I) at (6104);
%}
%\path[fill,name intersections={of = c4 and row5,by=h0d0U}](h0d0U)circle(1pt);
%\path[fill,name intersections={of = c4 and row6,by=h02d0U}](h02d0U)circle(1pt);
%{\tiny
%\coordinate[label=right:${h_2^2\omega U=h_0^2d_0U}$] (I) at (h02d0U);
%}
%\draw (d0U) -- (h02d0U);

\path[fill,name intersections={of = c1 and row3,by=0tauU}](0tauU)circle(1pt);
\path[fill,name intersections={of = c1 and row4,by=1tauU}](1tauU)circle(1pt);
\path[fill,name intersections={of = c1 and row5,by=2tauU}](2tauU)circle(1pt);
\path[fill,name intersections={of = c1 and row6,by=3tauU}](3tauU)circle(1pt);
\path[fill,name intersections={of = c1 and row7,by=4tauU}](4tauU)circle(1pt);
\path[fill,name intersections={of = c1 and row8,by=5tauU}](5tauU)circle(1pt);
\draw(0tauU)--(5tauU);
{\tiny
\coordinate[label=below:${\tau U}$] (I) at (0tauU);
}

\path[name path = c6](0.8,0)--(0.8,5.6);
\path[fill,name intersections={of = c6 and row4,by=484}](484)circle(1pt);
\path[fill,name intersections={of = c6 and row5,by=584}](584)circle(1pt);
\path[fill,name intersections={of = c6 and row6,by=684}](684)circle(1pt);
\path[fill,name intersections={of = c6 and row7,by=784}](784)circle(1pt);
\path[fill,name intersections={of = c6 and row8,by=884}](884)circle(1pt);
\draw(484)--(884);
{\tiny
\coordinate[label=right:${\omega x^2 U}$] (I) at (484);
}
{\tiny
\coordinate[label=right:${h^4_0\omega x^2 U}$] (I) at (884);
}
\path[fill,name intersections={of = c6 and row3,by=384}](384)circle(1pt);
{\tiny
\coordinate[label=above:${c_0x^2 U}$] (I) at (1.2,2);
}

%the line for t-s=13
%\path[fill,name intersections={of = c2 and row2,by=c2r2}](c2r2)circle(1pt);
%{\tiny
%\coordinate[label=left:${h_0h_2x^5 U}$] (I) at (c2r2);
%}
%\path[fill,name intersections={of = c1 and row1,by=139}](139)circle(1pt);

\path[fill,name intersections={of = c2 and row4,by=494}](494)circle(1pt);
{\tiny
\coordinate[label=right:${h_1c_0 x^2U}$] (I) at (494);
}

\path[fill,name intersections={of = c22 and row5,by=594}](594)circle(1pt);
{\tiny
\coordinate[label=left:${h_1\omega x^2U}$] (I) at (594);
}

\path[fill,name intersections={of = c2 and row6,by=6112}](6112)circle(1pt);

\path[fill,name intersections={of = c2 and row5,by=5112}](5112)circle(1pt);
{\tiny
\coordinate[label=right:${h_2\omega xU}$] (I) at (5112);
}

\path[fill,name intersections={of = c2 and row5,by=5112}](5112)circle(1pt);

\path[fill,name intersections={of = c2 and row6,by=6112}](6112)circle(1pt);
{\tiny
\coordinate[label=right:${h_0h_2\omega xU}$] (I) at (6112);
}
\path[fill,name intersections={of = c2 and row7,by=7112}](7112)circle(1pt);
{\tiny
\coordinate[label=above:${h^2_0h_2\omega xU}$] (I) at (7112);
}
\draw(5112)--(7112);

%the line of t-s=14
\path[fill,name intersections={of = c3 and row3,by=tauxU}](tauxU)circle(1pt);
{\tiny
\coordinate[label=below:${\tau xU}$] (I) at (tauxU);
}

\path[name path = c5](4,0)--(4,5.6);

%\draw[dashed] (tauxU) -- (5112);
\path[fill,name intersections={of = c3 and row4,by=1tauxU}](1tauxU)circle(1pt);
\path[fill,name intersections={of = c3 and row5,by=2tauxU}](2tauxU)circle(1pt);
\path[fill,name intersections={of = c3 and row6,by=3tauxU}](3tauxU)circle(1pt);
\path[fill,name intersections={of = c3 and row7,by=4tauxU}](4tauxU)circle(1pt);
\path[fill,name intersections={of = c3 and row8,by=5tauxU}](5tauxU)circle(1pt);
{\tiny
\coordinate[label=right:${h_0^5\tau xU}$] (I) at (5tauxU);
}
\draw (tauxU) -- (7,5.6);
%\draw[dashed] (1tauxU) -- (6112);
%\draw[dashed] (2tauxU) -- (7112);

\end{tikzpicture}
\renewcommand{\figurename}{Fig}
 \caption{$E_1^{s,m,n}$ ($t=m+n$)}
  \label{Fig.2}
\end{figure}

For $a\in \mm{Ext}_{\mathscr{A}_2}^{s,s+m}(\mathbb{Z}_2,\mathbb{Z}_2)$ and $b\in H^n(\mm{M}\xi;\mb{Z}_2)$, let $ab$ denote an element of $E_1^{s,m,n}$. We list some generators of the $E_1$-terms of the above AAHSS in Fig.\ref{Fig.2}. In particular, all generators of $E_1^{s,m,n}$ are shown when $m+n-s=13$. In Fig.\ref{Fig.2}, $xU$ and $x^2U$ denote the generators of $H^2(\mm{M}\xi;\mb{Z}_2)$ and $H^4(\mm{M}\xi;\mb{Z}_2)$ respectively. One dark dot in Fig.\ref{Fig.2} denotes an element of these $E_1$-terms, a vertical line $|$ between two dots denotes an $h_0$-action if the elements represented by the dots survive in the AAHSS.

Recall that the group $\Omega_{13}^{\mm{O}\langle 8\rangle}(\xi)$ depends on the bundle $\xi$ (cf. Section \ref{normal5}). Indeed, since $\mm{Sq}^iU=w_i(\xi)$, different bundle $\xi$ endows $H^\ast(\mm{M}\xi;\mb{Z}_2)$ with different $\mathscr A_2$-module structures, which are divided into four cases.

\begin{proposition}\label{E2term}
	In the AAHSS $E_1^{s,m,n}\Longrightarrow \mm{Ext}^{s,s+m+n}$, the elements, surviving to $E_\infty$, of $E_1^{s,m,n}$ are as follows where $m+n-s=13$
		\item Case 1: $w_2(\xi)=w_4(\xi)=0$. $h_1c_0x^2U,$ $h_0^sh_2\omega xU$ $(s=0,1)$. 
		%Only $c_0x^2U$, $\omega x^2U$, $h_1\omega x^2U$, and $h_0^2h_2\omega xU$ do not survive to $E_\infty$. Others . 
		\item  Case 2: $w_2(\xi)\ne 0$, $w_4(\xi)=0$. $h_0^sh_2\omega xU$ $(s=0,1,2)$.  %in addition, $h_1\omega x^2U=h_2\omega xU\in {Ext}^{s,s+13}$.
		\item Case 3: $w_2(\xi)=0$, $w_4(\xi)\ne 0$. $h_1c_0x^2U$.  %Only $c_0x^2U$, $\omega x^2U$, $h_1\omega x^2U$, and $h_0^sh_2\omega xU$ $(s=0,1,2)$ do not survive to $E_\infty$.
		 \item  Case 4: $w_2(\xi)\ne 0$, $w_4(\xi)\ne 0$. $h_0^sh_2\omega xU$ $(s=0,1,2)$.
		 
For all cases, $h_0^s\tau xU$ survives to $E_\infty$ in the AAHSS where $s\ge 0$.
\end{proposition}
\begin{proof}
	Let $d_r$ be the differential of the AAHSS
	$$d_r:E_r^{s,m,n}\to E_r^{s+1,m+r-1,n-r}$$
	All $E_1^{s,m,n}$-terms are shown in Fig.\ref{Fig.2} for $n+m=12,13$. 
	
	Case 1: By $\mm{Sq}^2(xU)=x^2U$, we have $d_2(h_1^s\omega x^2U)=h_1^{s+1}\omega xU$ for $s=0,1$. Since $h_0^2h_2\omega xU=h_1^3\omega xU$, $d_2(h_1^2\omega x^2U)=h_1^3\omega xU$. By  the degrees of differentials $d_r$, we have $E_\infty^{\ast,13-n,n}=E_3^{\ast,13-n,n}$.
	
	Case 2: By $\mm{Sq}^2(x^2U)=x^3U$, we have $d_2(c_0 x^3U)=h_1c_0 x^2U$ and $d_2(\omega x^3U)=h_1\omega x^2U$.
	
	Similarly, Cases 3 and 4 follow from the $\mm{Sq}^2$, $\mm{Sq}^4$ actions. By  the degrees of differentials $d_r$, $h_0^s\tau x U\in E_\infty^{3+s,14-n,n}$. 
\end{proof}

The Fig.\ref{Fig.2} and Proposition \ref{E2term} show ${Ext}^{s,13+s}_{\mathscr{A}_2}(H^\ast(\mm{M}\xi;\mathbb{Z}_2) , \mathbb{Z}_2)$ for $s\ge 0$ completely. Some nontrivial Adams differentials are given in the following lemma.
\begin{lemma}\label{d2ASS}
	In Cases $1,2,4$ of Proposition \ref{E2term}, $\tau xU$ represents a generator of $\mm{Ext}^{3,17}$, $h_2\omega xU$ of $\mm{Ext}^{5,18}$. Moreover, there is an Adams differential $d_2:\mm{Ext}^{3,17}\to \mm{Ext}^{5,18}$ so that 
	$d_2(\tau xU)= h_2\omega xU$.
\end{lemma} 
\begin{proof}
Let $\xi|2$ be the restriction bundle over $\mm{CP}^1$ of $\xi$, $\mm{M}\xi|_2$ be its Thom spectrum. On the one hand, by the AHSS 
 $$E_2^{p,q}(|_{2})=H_p(\mm{M}\xi|_2;\pi_{q}(\mm{MO} \langle 8  \rangle))\Longrightarrow\pi_{p+q}(\mm{MO}\langle 8\rangle\wedge \mm{M}\xi|_2)$$
we have $_2\pi_{13}(\mm{MO} \langle 8  \rangle\wedge \mm{M}\xi|_2)=0.$
On the other hand, 
for all cases as described in Proposition \ref{E2term}, ${Ext}^{5+s,18+s}_{\mathscr{A}_2}(H^\ast(\mm{M}\xi|_2;\mathbb{Z}_2) , \mathbb{Z}_2)$
is generated by $h_0^sh_2\omega xU$ ($s\ge 0$); ${Ext}^{3,17}_{\mathscr{A}_2}(H^\ast(\mm{M}\xi|_2;\mathbb{Z}_2) , \mathbb{Z}_2)$
is generated by $\tau xU$, and ${Ext}^{s,14+s}_{\mathscr{A}_2}(H^\ast(\mm{M}\xi|_2;\mathbb{Z}_2) , \mathbb{Z}_2)=0$ ($0\le s\le 2$). Thus $d_2(\tau xU)= h_2\omega xU$ in the $\mm{ASS}$ for the spectrum $\mm{MO}\langle 8\rangle\wedge \mm{M}\xi|_2$. By the naturality  and Proposition \ref{E2term}, the desired result follows.  	
\end{proof}  

Summarizing Fig.\ref{Fig.2}, Proposition \ref{E2term}, and Lemma \ref{d2ASS}, we have
\begin{lemma}\label{2prim}
	In Cases $1,3$ of Proposition \ref{E2term}, ${_2\Omega}_{13}^{\mm{O}\langle 8\rangle}(\xi)=\mb{Z}_2$.
	
	In Cases $2,4$ of Proposition \ref{E2term}, ${_2\Omega}_{13}^{\mm{O}\langle 8\rangle}(\xi)=0$.
\end{lemma}

Finally, we consider the $3$-primary part ${_3\Omega}_{13}^{\mm{O}\langle 8\rangle}(\xi)={\Omega}_{13}^{\mm{O}\langle 8\rangle}(\xi)/K^3$ where $K^3$ is a subgroup of all elements of finite order prime to $3$. 
\begin{lemma}\label{Sigma13}
(1) ${_3\Omega}_{13}^{\mm{O}\langle 8\rangle}(\xi)\cong \mb{Z}_3$ or $0$.

(2)	In case of ${_3\Omega}_{13}^{\mm{O}\langle 8\rangle}(\xi)\cong \mb{Z}_3$, its generator is represented by an exotic sphere $\Sigma^{13}$ (cf. Section \ref{intro}).
\end{lemma}
 \begin{proof}
 	By the AHSS, it is easy to get (1). Let $\iota:S^0\to \mm{M}\xi$ represent a generator of $\pi_0(\mm{M}\xi)\cong H_0(\mm{M}\xi)=\mb{Z}$. Then, we have the following morphism of spectral sequences 
 \[
\xymatrix@C=.9cm{
\bar{E}_2^{p,q}=H_p(S^0;\pi_q(\mm{MO}\langle 8\rangle))\ar@{=>}[r]\ar[d]&\pi_{p+q}(\mm{MO}\langle 8\rangle)\ar[d]^{\iota_\ast}\\
{E}_2^{p,q}=H_p(\mm{M}\xi;\pi_q(\mm{MO}\langle 8\rangle))\ar@{=>}[r]&\pi_{p+q}(\mm{MO}\langle 8\rangle\wedge \mm{M}\xi)
}
\]	
If ${E}_\infty^{0,13}={E}_2^{0,13}$, then a nontrivial element with order $3$ in  $\Omega_{13}^{\mm{O}\langle 8\rangle}(\xi)$ lies in the image of $\iota_\ast$. By \cite{KervaMilnor}, an exotic sphere $\Sigma^{13}$ is a generator of $\pi_{13}(\mm{MO}\langle 8\rangle)$. So, (2) follows.
 \end{proof} 

Now, we calculate ${_3\Omega}_{13}^{\mm{O}\langle 8\rangle}(\xi)$ clearly by the Adams spectral sequence. 
\begin{lemma}\label{p1mod3}
	If $p_1(\xi)\ne 0\mod 3$, ${_3\Omega}_{13}^{\mm{O}\langle 8\rangle}(\xi)=0$. 
\end{lemma}
\begin{proof}
 From \cite{Aikawa} and \cite{HoRa1995}, $\pi_{13}(\mm{MO}\langle 8\rangle)=\mb{Z}_3$ is generated by $$b_0h_{1,0}\in Ext_{\mc{A}}^{3,16}(H^\ast(\mm{MO}\langle 8\rangle;\mb{Z}_3),\mb{Z}_3)$$
in the ASS, where $\mc{A}$ is the mod $3$ Steenrod algebra. By the proof of (2) of Lemma \ref{Sigma13}, $b_0h_{1,0}U$ is nontrivial in $Ext_{\mc{A}}^{3,16}(H^\ast(\mm{MB};\mb{Z}_3),\mb{Z}_3)$ if ${_3\pi}_{13}(\mm{MB})=\mb{Z}_3$, where $U$ is the stable Thom class of $\mm{M}\xi$.
Note that
\begin{itemize}
	\item[(1)] $\mathcal{P}^1U=p_1(\xi)U\mod 3$  where $\mathcal{P}^1$ is the reduced Steenrod power in the mod $3$ Steenrod algebra $\mc{A}$ \cite{MilnorStasheff}.
	\item[(2)] $h_{1,0}$ is the dual of $\mc{P}^1$.
\end{itemize}
 If $p_1(\xi)\ne 0\mod 3$, then $\delta(b_0p_1(\xi)U)=b_0h_{1,0}U$ in the cobar complex of  $Ext_{\mc{A}}^{\ast,\ast}(H^\ast(\mm{MB};\mb{Z}_3),\mb{Z}_3)$. In other words, $b_0h_{1,0}U$ is trivial. Hence, the desired result follows from  (1) of Lemma \ref{Sigma13}.
\end{proof}

\begin{remark}
	The key to classify all spin manifolds $\mc{M}$ with {\it restriction lifts} lies in determining the bordism invariant of ${_2}\Omega_{13}^{\mm{O}\langle 8\rangle}(\xi)=\mb{Z}_2$ (cf. (1) of Lemma \ref{2prim}). I think that it is very difficult. 
\end{remark}

\begin{remark}
	If one want to classify all manifolds $\mc{M}$, can also set up the corresponding bordism groups $\Omega_{13}^{\mm{O}\langle 8\rangle}(-s\mathcal H)$. In other words, we can construct the space $\mm{BO}\langle 8\rangle\times \mm{CP}^\infty$ and the bundle $\gamma_8\times -s\mathcal H$ over it where $\mathcal H$ is the canonical complex line bundle over $\mm{CP}^\infty$, $s$ is the multiple of the first Pontryagin class of $\mc{M}$. However, the bordism group ${_2}\Omega_{13}^{\mm{O}\langle 8\rangle}(-s\mathcal H)$ is more complicated than Lemma \ref{2prim}. It is also tough to detect the bordism invariants of the  summands of ${_2}\Omega_{13}^{\mm{O}\langle 8\rangle}(-s\mathcal H)$.  %which implies that we need to characterize  
\end{remark}

\section{The sphere bundles of $8$-dimensional vector bundles}\label{homosec}
% A fibration $E\to B$ with fiber $F$ is called ``an $F$-fibration". 
 
 Let $\mm{G}_n$ be the monoid of all homotopy equivalences $S^{n-1}\to S^{n-1}$ topologized as the subspace of $\mm{Hom}(S^{n-1},S^{n-1})$ that contains all maps from $S^{n-1}$ to $S^{n-1}$. Then there exists a classifying space $\mm{BG}_n$ [\cite{Rudyak}, pp.232]. 
  Furthermore, there is a natural sequence 
$$\cdots\to\mm{BG}_{n-1}\to  \mm{BG}_n\stackrel{r_n^G}{\to } \mm{BG}_{n+1}\to \cdots$$
 where the homotopy fiber of $r_n^G : \mm{BG}_n \to  \mm{BG}_{n+1}$ is $(n - 2)$-connected [\cite{Rudyak}, Proposition 4.24]. Let $\mm{BG}$ be the telescope of this sequence. By taking the sphere bundle, we have a natural map $\alpha_{n}:\mm{BO}_n\to \mm{BG}_n$, which induces the map $\alpha:\mm{BO}\to \mm{BG}$. The classifying spaces $\mm{BO}$ and $\mm{BG}$ for the stable bundle theories are infinite loop spaces and the natural map between them is an infinite loop map by theorems of Boardman and Vogt \cite{BoaVog}. In fact, they are homotopy equivalent to the $0$th spaces in the $\Omega$-spectra of the multiplicative cohomology theories $KO$ and $KG$ respectively. Thus, there is a natural map between the $\Omega$-spectra $KO$ and $KG$, which is induced by $\alpha:\mm{BO}\to \mm{BG}$. In addition, the homotopy fiber $G/O$ of $\alpha:\mm{BO}\to \mm{BG}$ is also an infinite loop space. 

\begin{lemma}\label{KOtoKG}
The following items hold:
\item (1) $\widetilde{ KO}^0(\mm{CP}^3)\cong[\mm{CP}^3,\mm{BO}_8]$.
\item (2) $\widetilde{ KG}^0(\mm{CP}^3)\cong [\mm{CP}^3,\mm{BG}_8]$.

\item (3) $\widetilde{ KO}^0(\mm{CP}^3)\to\widetilde{  KG}^0(\mm{CP}^3)$ is an epimorphism. 
\end{lemma} 
\begin{proof}
	 Let $V=O$, $G$. The homotopy fiber of
 $r_{8\ast}^V : \mm{BV}_8 \to \mm{BV}$ is $6$-connected at least. Then, it is easy to get (1) and (2). 
 
 It is well known that $\pi_1(G/O)=\pi_3(G/O)=\pi_5(G/O)=0$ \cite{Sull}. By the obstruction theory, any map from $\mm{CP}^3$ to $\mm{BG}$ can be lifted to $\mm{BO}$. Thus, (3) follows.   
 \end{proof}

%By Sanderson [\cite{Sand}, Theorem 3.9], we have $\widetilde{ KO}^0(\mm{CP}^2)=\mb{Z}$ and $\widetilde{ KO}^0(\mm{CP}^3)\cong\mb{Z}$. Furthermore, we have
\begin{lemma}\label{firstPonofspherebundle}
	The stable class of a real vector bundle over $\mm{CP}^3$ is determined by its first Pontrjagin class. Moreover, for every integer $n$, there exists a real vector bundle $\eta$ over $\mm{CP}^3$ so that $p_1(\eta)=nx^2$ where $x\in H^2(\mm{CP}^3)$ is a generator.  
\end{lemma}
\begin{proof}
	By Sanderson [\cite{Sand}, Theorem 3.9], $\widetilde{ KO}^0(\mm{CP}^3)\cong\mb{Z}$ is generated by the stable class of the underlying real vector bundle, namely $r(h_c)$ of the canonical complex line bundle $h_c$ over $\mm{CP}^3$. Since that the total Chern class of $h_c$ is $c(h_c)=1+x$, the total Pontrjagin class of the stable class of $r(h_c)$ is $p=1+x^2$ where $x\in H^2(\mm{CP}^3)$ is a generator. By the formula for the total Pontrjagin class of the Whitney sum of real vector bundles, we finish this proof.  
\end{proof}

Finally, we consider when the sphere bundles of two $8$-dimensional vector bundles over $\mm{CP}^3$ are (strongly) fiber homotopy equivalent (see [\cite{Sw}, Definition 15.10]).

%present a sufficient condition to 
%characterize the ce of . 
%determine whether  are t. 

\begin{lemma}\label{homotopy}
Let $\eta$, $\eta^\prime$ be arbitrary $8$-dimensional real vector bundles over $\mm{CP}^3$. The sphere bundle of $\eta$ is fiber homotopy equivalent to the sphere bundle of $\eta^\prime$ if $p_1(\eta)=p_1(\eta^\prime)\mod 24$.
\end{lemma}
\begin{proof} 
For the cofibration sequence [\cite{Hatcher}, pp.398] 
$S^2\stackrel{i}{\to} \mm{CP}^2\to S^4 \to\Sigma S^2$ %where $i:S^2\to \mm{CP}^2$ is the embedding, $S^4=\mm{CP}^2/S^2$, $\Sigma$ is the suspension, the map $S^4\to \Sigma S^2$ is the composition of a homotopy inverse of the quotient map $\mm{CP}^2\cup_i D^3\to \mm{CP}^2/S^2=S^4$ followed by the map  
we have the following exact sequences 
\[
\xymatrix@C=.9cm{
0\ar[d]\ar[r]&[S^4,\mm{BO}]\ar[d]^{\alpha_4}\ar[r]&[\mm{CP}^2,\mm{BO}]\ar[d]\ar[r]&[S^2,\mm{BO}]\ar[r]\ar[d]^{\alpha_2}&0\ar[d]\\
[\Sigma S^2,\mm{BG}]\ar[r]^{\rho_3}&[S^4,\mm{BG}]\ar[r]&[\mm{CP}^2,\mm{BG}]\ar[r]&[S^2,\mm{BG}]\ar[r]^{\rho_2}&\cdots
}
\]
Note that $\alpha_2$ is an isomorphism, $\alpha_4:\mb{Z}\to \mb{Z}_{24}$ is an epimorphism \cite{Sull}, and $[\mm{CP}^2,\mm{BO}]=\mb{Z}$ [\cite{Sand}, Theorem 3.9]. By \cite{WangX}, $\rho_3:\mb{Z}_2\to \mb{Z}_{24}$ is a monomorphism. Therefore, $[\mm{CP}^2,\mm{BO}]=\mb{Z}\to [\mm{CP}^2,\mm{BG}]=\mb{Z}_{24}$ is an epimorphism. By the exact sequences for the cofibration sequence $\mm{CP}^2\to \mm{CP}^3\to S^6\to \Sigma\mm{CP}^2$, we have that $[\mm{CP}^3,\mm{BO}]=\mb{Z}\to [\mm{CP}^3,\mm{BG}]=\mb{Z}_{24}$ is an epimorphism. So, the desired result follows.  
\end{proof}

\section{Proofs of Theorems \ref{1.1} and \ref{1.2}}\label{proofmain}

Proofs of (1), (2) and (3) of Theorem \ref{1.1}:
	For any nonspin manifold $\mc{M}$ with a {\it restriction lift}, the corresponding bordism group $\Omega_{13}^{\mm{O}\langle 8\rangle}(\xi)$ as in Section \ref{normal5} satisfies the case $2$ or $4$ as in Lemma \ref{2prim}. 
All manifolds $\mc{M}^\prime$ with $p_1(\mc{M}')=p_1(\mc{M})$ lie in the same bordism group. 

By Lemma \ref{Sigma13} and \ref{hcobordism}, $\mc{M}$ is diffeomorphic to $\mc{M}^\prime\#\Sigma^{13}$ for some homotopy sphere $\Sigma^{13}$. Especially, by Lemma \ref{p1mod3}, $\mc{M}$ is diffeomorphic to $\mc{M}^\prime$ if $p_1(\mc{M})\ne 0\mod 3$.

Recall the definition of $\#$.
By Poincar\'e Theorem, $\Sigma^{13}$ is homeomorphic to $S^{13}$. Thus $\mc{M}^\prime\#\Sigma^{13}$ is homeomorphic to 
$\mc{M}^\prime\# S^{13}$. Since $\mc{M}^\prime\# S^{13}= \mc{M}'$, $\mc{M}$ is homeomorphic to $\mc{M}^\prime$. $\Box$ 

Proof of Corollary \ref{1.2}:
By (1) of Lemma \ref{KOtoKG} and Lemma \ref{firstPonofspherebundle}, we have that for every $k\in \mb{Z}$, there exists an $8$-dimensional vector bundle $\eta$ over $\mm{CP}^3$ so that $p_1(\eta)=kx^2$.
For the sphere bundle $S^n\to \overline{\mc{M}}\stackrel{q}{\to} \mm{CP}^3$ of a vector bundle $\eta$, it is well-known that $\tau(\overline{\mc{M}})\oplus \epsilon^1\cong q^\ast(\tau(\mm{CP}^3)\oplus \eta)$ where $\tau$ denotes the tangent bundle, $\epsilon^1$ is the $1$-dimensional trivial bundle over $\overline{\mc{M}}$. Hence $p_1(\overline{\mc{M}})=(4+k)x^2$. By (1) and (2) of Theorem \ref{1.1}, Corollary \ref{1.2} follows. $\Box$

Proof of (4) of Theorem \ref{1.1}:
Note that $\mc{M}$ is homeomorphic to $\mc{M}\# \Sigma^{13}$. 
 By Lemma \ref{homotopy} and Corollary \ref{1.2}, $\mc{M}$ is homotopy equivalent to $\mc{M}^\prime$ if $p_1(\mc{M})=p_1(\mc{M}^\prime)\mod 24$. Furthermore, $p_1$ mod $24$ is a homotopy invariant \cite{Hi}. $\Box$

\section{A metric of non-negative sectional curvature}\label{metricSection}

In Riemannian geometry, there is an interesting question:
\begin{question}
	Let $E$ be the total space of a vector bundle over a compact manifold $B$ with non-negative sectional curvature. Does E admit a metric of non-negative curvature?
\end{question} 

From \cite{GAD}, we know that: for arbitrary real (resp. complex) vector bundle $\eta$ over $\mm{CP}^3$, there exists $k\ge 0$ such that the total space of $\eta\oplus \epsilon^k$ admits a metric of non-negative curvature where $\epsilon^k$ is the $k$-dimensional trivial bundle. In this section, we will prove that the total spaces of certain real (resp. complex) vector bundles over $\mm{CP}^3$ admit metrics of non-negative curvature.

 Recall that $\mm{CP}^3=\mm{U}(4)/(\mm{U}(1)\times \mm{U}(3))$. Now we consider the following bundle, namely $\eta_{\bold q}$ ($\bold q=(q_1,q_2,q_3,q_4)$, $q_i\in \mb{Z}$)
 $$\eta_{\bold q}:\mb{C}^4\to \mm{U}(4)\times_{(\mm{U}(1)\times \mm{U}(3))}\mb{C}^4\to \mm{CP}^3$$
 where the $\mm{U}(1)\times \mm{U}(3)$ action on $\mb{C}^4$ is given by, for $(z,A)\in \mm{U}(1)\times \mm{U}(3) $, 
 $$(z,A)\curvearrowright (c_1,c_2,c_3,c_4):=(z^{q_1}c_1,z^{q_2}c_2,z^{q_3}c_3,z^{q_4}c_4)$$
We can endow $\mm{U}(4) \times \mb{C}^4$ with the product metric of $\langle , \rangle_{\mm{U}(4)}$ and the standard Euclidean metric. By O'Neill's Theorem on Riemannian submersions, $\mm{U}(4)\times_{(\mm{U}(1)\times \mm{U}(3))}\mb{C}^4$ admits a metric of non-negative curvature. On the other hand, by the $\mm{U}(1)\times \mm{U}(3)$ action on $\mb{C}^4$, we have $\mm{U}(4)\times_{(\mm{U}(1)\times \mm{U}(3))}\mb{C}^4=S^7\times_{S^1}\mb{C}^4$ where $\mm{U}(1)=S^1$. Moreover, the bundle $\eta_{\bold q}$ has the form
$$\mb{C}^4\to S^7\times_{S^1}\mb{C}^4\to \mm{CP}^3.$$ 
\begin{lemma}\label{metricon8vb}
	For every integer $n\ge 0$, there exists $\bold q=(q_1,q_2,q_3,q_4)$ such that the first Pontrjagin class of the underlying real bundle of $\eta_{\bold q}$ equals to $nx^2$ where $x\in H^2(\mm{CP}^3)$ is a generator.
\end{lemma} 
\begin{proof}
We first present a bundle, namely $\theta_{\bold q}$ over $\mm{CP}^\infty$
$$\mb{C}^4\to S^\infty\times_{S^1}\mb{C}^4\to \mm{CP}^\infty$$
where the $S^1$ action on $\mb{C}^4$ is given by, for $z\in S^1$, $$z\curvearrowright (c_1,c_2,c_3,c_4):=(z^{q_1}c_1,z^{q_2}c_2,z^{q_3}c_3,z^{q_4}c_4).$$   
By the natural inclusion $i:\mm{CP}^3\to \mm{CP}^\infty$, we have $\eta_{\bold q}=i^\ast\theta_{\bold q}$.

Note that $\theta_{\bold q}$ is the associated bundle of a principal $\mm{U}(4)$-bundle, namely $\Theta_{\bold q}:\mm{U}(4)\to S^\infty\times_{S^1} \mm{U}(4)\to \mm{CP}^\infty$, i.e. there is a bundle isomorphism so that $(S^\infty\times_{S^1} \mm{U}(4))\times_{\mm{U}(4)}\mb{C}^4\cong S^\infty\times_{S^1} \mb{C}^4$ where the $\mm{U}(4)$ action on $\mb{C}^4$ is standard, the $S^1$ action on $\mm{U}(4)$ is induced by the embedding $S^1\to \mm{U}(4):z\to \mm{diag}(z^{q_1},z^{q_2},z^{q_3},z^{q_4})$. 

Let $\kappa_{\bold q}:\mm{CP}^\infty \to \mm{BU}(4)$ be the classifying map of $\Theta_{\bold q}$. By the following homotopy fibration,
$$S^1\to \mm{U}(4)\to \mm{U}(4)/S^1\simeq S^\infty\times_{S^1} \mm{U}(4)  \to \mm{B}S^1=\mm{CP}^\infty \stackrel{\kappa_{\bold q}}{\longrightarrow }\mm{BU}(4)$$
${\kappa}_\bold{q}$ is induced by the above embedding $S^1\to  \mm{U}(4).$ 

Now we recall the classical characteristic theory for the sake of convenience.
Let $G$ be a Lie group, and $T^n \subset G$ be a maximal torus in $G$ with induced map $\kappa: \mm{BT}^n\to \mm{BG}$. Let $I_G$ be the ring of polynomials in $H^\ast(\mm{BT}^n)$ invariant under the Weyl group $W(G)$.
\begin{theorem}
[Borel's Theorem \cite{Borel1953}] If $H^\ast(G)$ and $H^\ast(G/T^n)$ are torsion-free, then $\kappa^\ast : H^\ast(\mm{BG}) \to H^\ast(\mm{BT}^n)$ is a monomorphism with range $I_G$.
\end{theorem}

As was shown in \cite{Borel1953}, the conditions of Borel's Theorem are satisfied for all classical groups.
Recall that $H^\ast(\mm{BU}(4))\cong \Bbb Z[c_1, \cdots, c_4]$ a polynomial rings on the Chern classses $c_1, \cdots, c_4$. Applying Borel's Theorem to $\mm{U}(4)$, we have that
%where the standard maximal torus is
%$$T^5=\{\mm{diag}(e^{i\theta_0},\cdots ,e^{i\theta_5})\quad |\quad \theta_0+\cdots +\theta_5=0 \mod (2\pi)\}.$$
the homomorphism
$H^\ast(\mm{BU}(4))\to H^\ast(\mm{BT}^4)$
sends  $c_i$ to the elementary symmetric polynomial $\sigma_i(x_1,\cdots,x_4)$ of degree $i$ where $x_j\in H^2(\mm{BT}^4)$, $1\le j\le 4$, are  the generators.
Given a circle subgroup $S^1$ in the maximal torus parameterized $\bold q$ as above, $\kappa_{\bold q}^*(c_i)=\sigma_i(q_1,\cdots,q_4)x^i$ where $x\in H^2(\mm{B}{S}^1)\cong \Bbb Z$ is a  generator. 

Therefore, the first and second Chern classes of the complex vector bundle $\eta_{\bold q}$ are $c_1(\eta_{\bold q})=\sigma_1(q_1,\cdots,q_4)x$, $c_2(\eta_{\bold q})=\sigma_2(q_1,\cdots,q_4)x^2$. Thus, the first Pontryagin class of the underlying real bundle of $\eta_{\bold q}$ equals to $p_1=(\sigma_1(q_1,\cdots,q_4))^2-2\sigma_2(q_1,\cdots,q_4)=\Sigma_{i=1}^4q_i^2$.

Recall Lagrange's four-square theorem: every integer $n\ge 0$, there exist $a,b,c,d\in \mb{Z}$ so that $n=a^2+b^2+c^2+d^2$, which has been proved by Lagrange and Euler respectively. This finishes our proof.
\end{proof}
 
By (1) of Lemma \ref{KOtoKG}, Lemma \ref{firstPonofspherebundle} and Lemma \ref{metricon8vb}, we have
\begin{proposition}\label{metric}
	For arbitrary $8$-dimensional real vector bundle $\eta$ over $\mm{CP}^3$ with $p_1(\eta)=nx^2$, $n\ge 0$, its total space admits a metric of non-negative sectional curvature. Moreover, the total space of the sphere bundle of $\eta$ admits a metric of non-negative sectional curvature.
\end{proposition}

Proof of Theorem \ref{1.3}: 
Assume that $\mc{M}$ is nonspin, and admits a {\it restriction lift}, $p_1(\mc{M})=(n+4)x^2$ where $n\ge 0$. By (1) of Lemma \ref{KOtoKG} and Lemma \ref{firstPonofspherebundle}, there exists an $8$-dimensional vector bundle $\eta$ over $\mm{CP}^3$ with $p_1(\eta)=nx^2$ where $n\ge 0$. 
 Let $\overline{\mc{M}}$ be the total space of its sphere bundle. Then $p_1(\overline{\mc{M}})=(n+4)x^2$. By (1), (2) of Theorem \ref{1.1} and Proposition \ref{metric}, we finish the proof. $\Box$

\end{document}